\documentclass[12pt]{amsart}

\usepackage[english]{babel}
\usepackage[utf8x]{inputenc}
\usepackage[T1]{fontenc}
\usepackage{amsmath,amssymb,amsthm,mathtools,mathdots,nicefrac,tikz,bbding,stmaryrd,float,multicol,caption,bbm, enumitem,graphicx,comment}

\usepackage[mathscr]{euscript}
\usepackage[a4paper,top=3cm,bottom=2cm,left=3cm,right=3cm,marginparwidth=1.75cm]{geometry}

\usepackage[colorlinks=true, allcolors=blue]{hyperref}

\theoremstyle{definition} 
\newtheorem{theorem}{Theorem}[section]

\newtheorem{lemma}[theorem]{Lemma}
\newtheorem{proposition}[theorem]{Proposition}
\newtheorem{definition}[theorem]{Definition}
\newtheorem{example}[theorem]{Example}
\newtheorem{remark}[theorem]{Remark}

\newtheorem{manualtheoreminner}{Theorem}
\newenvironment{manualtheorem}[1]{%
  \renewcommand\themanualtheoreminner{#1}%
  \manualtheoreminner
}{\endmanualtheoreminner}

\newtheorem{manualdefinitioninner}{Property}
\newenvironment{manualdefinition}[1]{%
  \renewcommand\themanualdefinitioninner{#1}%
  \manualdefinitioninner
}{\endmanualdefinitioninner}

\newcommand{\interior}[1]{%
  {\kern0pt#1}^{\mathrm{o}}%
}

\newcommand{\Q}{\mathbb{Q}}

\newcommand{\Spec}{\text{Spec}}
\newcommand{\Ass}{\text{Ass}}

\newcommand{\height}{\text{ht}}
\newcommand{\ann}{\text{ann}}

\begin{document}

\title{Gluing associated prime ideals of small height}
\author{S. Loepp and Liz Ostermeyer}

\maketitle

\begin{abstract}
    Let $B$ be a local (Noetherian) ring and suppose that  $B$ has $n$ associated prime ideals where $n \geq 2$. We identify sufficient conditions for there to exist a local (Noetherian) subring $S$ of $B$ such that $S$ and $B$ have the same completion and $S$ has exactly $n - 1$ associated prime ideals. We include applications and consequences of this result.
\end{abstract}

%

\section{Introduction}

Prime ideals are, of course, fundamental objects in commutative algebra. Studying the set of prime ideals of a ring $R$ (i.e. its prime spectrum) provides insight into the ring itself. 
%
%
%
In this paper we focus on the set of associated prime ideals of a local ring. 
When we say a ring is \emph{local}, we mean that it is Noetherian and has a unique maximal ideal. Since local rings are Noetherian, the set of associated prime ideals of a local ring is finite. 

It is worth noting that, given a Noetherian ring $R$, a prime ideal $Q$ is an associated prime ideal of $R$ if and only if $Q$ is the annihilator of some element $r \in R$. This characterization of an associated prime ideal provides insight into one of the reasons for studying associated prime ideals, which is that, for a ring $R$, the set of zero divisors of $R$ is exactly the union of the associated prime ideals of $R$. In other words, understanding the set of associated prime ideals of a ring gives us insight into the set of zerodivisors of a ring.  Additionally, the set of minimal prime ideals of a ring are a key component of the overall structure of the prime spectrum, and because all minimal prime ideals are associated prime ideals, our study of associated prime ideals has direct consequences for understanding prime spectra, more generally. This insight can inform our understanding of a ring's structure. 

Before we address the results of this paper in further detail, we make a few remarks on notation and convention. All rings referenced here will be commutative rings with unity. If a ring has a unique maximal ideal but is not necessarily Noetherian, we call it a quasi-local ring. We denote a local (or a quasi-local) ring $R$ with maximal ideal $M$ as $(R, M)$. 
The annihilator of an element $r$ in the ring $R$ will be denoted by $\ann_R(r)$. The prime spectrum of $R$ will be denoted $\Spec(R)$, while the set of associated prime ideals of $R$ will be denoted $\Ass(R)$. If $(R,M)$ is a local ring, then we use $\widehat{R}$ to denote the $M$-adic completion of $R$.

We now provide more details regarding our main result. Theorem \ref{bigassthm}, which we formally state after providing the necessary background, can be applied to a certain class of local rings $B$ whose associated prime ideals have sufficiently small height. Given such a $B$, and letting $\{Q_1,Q_2, \ldots, Q_n\}$ where $n \geq 2$ be the set of associated prime ideals of $B$, we can select two associated prime ideals of $B$ satisfying certain conditions.  Without loss of generality, call these two associated prime ideals $Q_1$ and $Q_2$. In Theorem \ref{bigassthm}, we show the existence of a local subring $S$ of $B$ such that $S$ has exactly $n - 1$ associated prime ideals and $S \cap Q_1 = S \cap Q_2$. In addition, the theorem ensures that $S$ and $B$ have the same completion. In Section \ref{Applications} we show that, because of this, the set of associated primes of $S$ is exactly $\{ S \cap Q_1 = S \cap Q_2 , S \cap Q_3, \dots, S \cap Q_n\}$. We can therefore think of $S$ as a subring of $B$ in which the associated prime ideals $Q_1$ and $Q_2$ are ``glued'' together. Although our main result only shows we can ``glue'' two associated prime ideals of $B$, it is natural to wonder if we can repeat this gluing operation. In fact we can, and in Section \ref{Applications} we demonstrate that this process can be iterated to form a chain 
%
of local rings $S_1 \supseteq S_2 \supseteq \dots \supseteq S_{n}$ such that, if $\{Q_1, \dots, Q_{n}\}$ are the associated prime ideals of $S_1$, then $S_i$ has exactly $n-i + 1$ associated prime ideals. We also show that, because $S_i$ and $S_j$ have the same completion for any $i,j \in \{1, 2, \dots, n \}$, 
\[\Ass(S_i)=\{S_i \cap Q_1 = S_i \cap Q_2 = \dots = S_i \cap Q_{i}, S_i \cap Q_{i+1}, \dots, S_i \cap Q_{n}\}.\] As $i$ increases, $S_i$ has fewer associated prime ideals, and hence, fewer prime ideals consisting of exclusively zerodivisors.  Therefore, one could think of this chain in some sense as a chain of subrings of $S_1$ that are getting ``closer and closer'' to being an integral domain.  In Section \ref{Applications}, we show that the last ring in the chain, $S_{n}$, is, in fact, an integral domain.

In constructing such a chain of subrings, we are proving the existence of certain associated prime ideal structures of subrings of a given local ring. In particular, this provides insight into the relationship between the set of associated prime ideals of a local ring and the set of associated prime ideals of its completion. We discuss this in more detail in Section \ref{Applications}.

The result presented in this article is a generalization of a result in \cite{mg}. In particular, the following is a consequence of Theorem 2.14 in \cite{mg}. Let $(B, M)$ be a reduced local ring containing the rationals such that $B/M$ is uncountable and $|B| = |B/M|$, and let $\Ass(B) = \{Q_1,Q_2, \ldots ,Q_n\}$ with $n \geq 2$. Then there exists a reduced local ring $(S, S \cap M)$ such that 
\medskip
\begin{enumerate}
    \item $S$ contains the rationals
    \item $ S \subseteq B $
    \item $\widehat{B}=\widehat{S}$
    \item $S/(S \cap M)$ is uncountable and $|S| = |S/(S \cap M)|$,
    \item $S \cap Q_1 = S \cap Q_2$.
\end{enumerate}
\medskip

Our goal is to show a more general version of this result. In addition to the properties mentioned above, the ring $S$ in Theorem 2.14 in \cite{mg} also satisfies the property that there is a containment-preserving bijection from the prime ideals of $B$ of positive height and the prime ideals of $S$ of positive height. As $B$ is reduced, this means that there is a containment-preserving bijection from the prime ideals of $B$ that are not associated prime ideals and the prime ideals of $S$ that are not associated prime ideals. Since we are not concerned in this article with prime ideals that are not associated, we are able to drop some of the hypotheses on the ring $B$. 


The following theorem is the main result of this paper.

\begin{manualtheorem}{\ref{bigassthm}}
Let $(B,M)$ be a local ring containing the rationals and assume that $B/M$ is uncountable. Suppose that $\Ass(B) = \{Q_1, Q_2, \dots, Q_n\}$, with $n \geq 2$ and that $M \not\in \Ass(B)$. Suppose also $\height(Q_i) \leq 1$ for all $i \in \{1, 2, \dots, n\}$. Consider $Q_1$ and $Q_2$ such that $\height(Q_1)=\height(Q_2)=0$ or $\height(Q_1)=1, \height(Q_2)=0$ and $Q_2$ is the unique minimal prime ideal contained in $Q_1$ \footnote{In Section \ref{sec3}, we define these conditions on $Q_1$ and $Q_2$ as property \ref{*}.}. Then there is a local ring $S \subseteq B$ with maximal ideal $S \cap M$ such that 
\medskip
\begin{enumerate}
    \item $S$ contains the rationals
    \item $\widehat{S}=\widehat{B}$
    \item $S/(S \cap M)$ is uncountable
    \item $S \cap Q_i = S \cap Q_j$ if and only if $i=j$ or $i, j \in \{1, 2\}$.
\end{enumerate}
\medskip
\end{manualtheorem}



The following are examples of rings where Theorem \ref{bigassthm} in this paper can be applied but Theorem 2.14 in \cite{mg} cannot.

\begin{example}
Let $B = \mathbb{R}[[X_1, X_2, \dots, X_n]]/(\langle X_1^2 \rangle \cap \langle X_2^2 \rangle \cap \dots \cap \langle X_n^2 \rangle)$ with $n \geq 2$. Note that $B$ satisfies the conditions of Theorem \ref{bigassthm} in this paper, but is not reduced, and so does not satisfy the conditions of Theorem 2.14 in \cite{mg}. In this example, all associated prime ideals are minimal prime ideals and so any pair of associated prime ideals can be chosen as $Q_1$ and $Q_2$.
\end{example}


\begin{example} \label{2 stix}
Let $B=\mathbb{C}[[X_1,X_2,X_3,X_4]]/(\langle X_1^2, X_1X_2\rangle \cap \langle X_3^2, X_3X_4 \rangle)$. Then $B$ satisfies the conditions of Theorem \ref{bigassthm} and has two minimal prime ideals and two embedded associated prime ideals. Using Theorem \ref{bigassthm}, one could choose $Q_1$ and $Q_2$ to be the two minimal prime ideals, and one could choose $Q_2$ to be minimal and $Q_1$ to be the height one associated prime ideals containing $Q_2$. We show in Section \ref{Applications} (see Proposition \ref{ht1-no-no}) that if $\height(Q_1)=\height(Q_2)=1$ there is no subring $S$ of $B$ having the same completion as $B$ satisfying the property that $S$ has exactly three associated prime ideals and $\Ass(S) = \{Q_1 \cap S = Q_2 \cap S, Q_3 \cap S, Q_4 \cap S\}.$
\end{example}

\begin{example}
Let 
\[B = \mathbb{R}[[X_1, X_2, X_3, X_4]]/(\langle X_1^2 \rangle \cap \langle X_2^2 \rangle \cap \langle X_1, X_2 \rangle^5 \cap \langle X_3^2 \rangle \cap \langle X_4^2 \rangle \cap \langle X_3, X_4 \rangle^5)\]

\noindent As in the previous two examples, $B$ satisfies conditions for Theorem \ref{bigassthm}. By Theorem \ref{bigassthm}, we know there exists a subring $S$ of $B$ such that $S \cap \langle X_1 \rangle = S \cap \langle X_3 \rangle$. $S$ also satisfies the necessary conditions to apply Theorem \ref{bigassthm} again, and so we know there exists an $S'$ where $S' \cap \langle X_2 \rangle = S' \cap \langle X_4 \rangle$. 
These two ``gluing'' moves generate a ring $S'$ that has four associated prime ideals, two that are minimal and two that have height one.  Moreover, the two minimal prime ideals are both contained in the height-one associated prime ideals.

\end{example}

\section{Generalized Gluing Theorem}\label{sec3}

We are now ready to begin the proof of our main result, The Generalized Gluing Theorem.  The proof is directly inspired by techniques from \cite{mg}. Throughout, $(B,M)$ will be a local ring with $B/M$ uncountable. We begin our construction with a definition, which is an amended version of a Minimal-Gluing subring, introduced in \cite{mg}. The proof of Lemma 2.8 in \cite{mg} uses the fact that, if $R$ is a subring of $B$, and $Q_1$ is a minimal prime ideal of $B$, then for all $a \in R \cap Q_1, \ann_B(a) \nsubseteq Q_1$. This fact follows from the assumption used in \cite{mg} that $B$ is reduced. When dropping that assumption, however, we cannot conclude that for all $a \in R \cap Q_1, \ann_B(a) \nsubseteq Q_1$. In order to prove the analogous version of Lemma 2.8 in \cite{mg}, we add this condition to the definition of a Minimal-Gluing subring. The result of the amendment is the following new definition. 

%
\begin{definition}\label{omg}
Let $(B, M)$ be a local ring with $B/M$ uncountable such that $\Ass(B) = \{Q_1, Q_2,\dots,Q_n\}$ with $n \geq 2$. A quasi-local subring $(R, R \cap M)$ of $B$ is called a One–or-Minimal-Gluing subring of $B$, or an OMG-subring of $B,$
 if $R$ is infinite, $|R| < |B/M|$, $R \cap Q_1 = R \cap Q_2$, and for all $a \in R \cap Q_1, \ann_B(a) \nsubseteq Q_1$ and $\ann_B(a) \nsubseteq Q_2$. 
\end{definition}

This definition describes a subring $R$ of $B$ such that two distinct associated prime ideals of $B$, call them $Q_1$ and $Q_2$, satisfy $Q_1 \cap R = Q_2 \cap R$. Intuitively, we think of $Q_1$ and $Q_2$ as having been ``glued together'' in the subring $R$ of $B$. 

In Definition \ref{2xomg}, we introduce a special type of OMG-subring, called an organized OMG-subring, in which exactly two associated prime ideals of $B$ have been glued together and all others remain distinct from each other. The process for constructing an organized OMG-subring is closely drawn from techniques in \cite{mg}. 

We start, in our case, with an OMG-subring. We assume that $B$ contains the rationals in order to guarantee there exists an OMG-subring of $B$. With this OMG-subring, we adjoin elements to construct an organized OMG-subring of $B$.  Then we construct increasingly larger subrings of $B$, carefully selecting the elements we adjoin in order to ensure the subring at each step is an OMG-subring. When an OMG-subring contains an organized OMG-subring, it is itself an organized OMG-subring.  Since our OMG-subrings contain an organzied OMG-subring, they will also be organized OMG-subrings. The following two lemmas will be useful in choosing which elements can be safely adjoined to maintain an OMG-subring as we adjoin elements. The first of the two lemmas can be thought of as a generalization of the prime avoidance theorem.

%
\begin{lemma}[\cite{heitmann}, Lemma 3]\label{primeavoidbig}

Let $(B, M)$ be a local ring. Let $C \subseteq \Spec(B),$ let $I$ be an ideal of $B$ such that $I \nsubseteq P$ for every $P \in C$, and let $D$ be a subset of $B$. Suppose $|C \times D| < |B/M|$. Then $I \nsubseteq \bigcup\{P +r | P \in C, r \in D\}$.

\end{lemma}

The next lemma provides a sufficient condition on $x \in B$ such that, given an OMG-subring $R$ of $B$, $R[x]_{(R[x] \cap M)}$ is an OMG-subring of $B$.

%
\begin{lemma}\label{transcend}

Let $(B, M)$ be a local ring with $B/M$ uncountable. Suppose that $\Ass(B)= \{Q_1, \dots, Q_n\}$ with $n \geq 2$. Suppose $(R, R\cap M)$ is an OMG-subring of $B$. If $x \in B$ satisfies the condition that $x + Q_i \in B/Q_i$ is transcendental over $R/(Q_i \cap R)$ for $i \in \{1,2\},$ then $S= R[x]_{(R[x]\cap M)}$ is an OMG-subring of $B$ with $|S|=|R|.$
\end{lemma}

\begin{proof}
Much of our proof follows the proof of Lemma 2.3 in \cite{mg} exactly.

Since $R$ is infinite, $|S|=|R|$ and so $|S| < |B/M|$. Now suppose $f \in R[x] \cap Q_1$. Then $f=r_mx^m + r_{m-1}x^{m-1}+ \dots + r_1x+r_0 \in Q_1$ where $r_j \in R$ for $0 \leq j \leq m$. Since $x + Q_1$ is transcendental over $R/(R \cap Q_1)$, we have $r_j \in R \cap Q_1=R \cap Q_2$. Hence, $f \in Q_2$, and so $R[x] \cap Q_1 \subseteq R[x] \cap Q_2$. Similarly, $R[x] \cap Q_2 \subseteq R[x] \cap Q_1$, and therefore $R[x] \cap Q_1 = R[x] \cap Q_2$. It follows that $S \cap Q_1 = S \cap Q_2.$

We now show that, for all $f \in R[x] \cap Q_1$, $\ann_B(f) \nsubseteq Q_1$ and $\ann_B(f) \nsubseteq Q_2$. Suppose $f = r_0 + r_1x + r_2x^2 + \dots + r_mx^m \in Q_1.$ Then, as before, $r_i \in Q_1$ for all $i \in \{0, 1, \dots, m\}$. Because $R$ is an OMG-subring, we have $\ann_B(r_i) \nsubseteq Q_1$ for all $i = 0,1,2, \ldots ,m$. Hence, for all $i$ there exists $a_i \in \ann_B(r_i)$ such that $a_i \notin Q_1$. Let $\alpha = a_0a_1\cdots a_m$. Then $\alpha \in \ann_B(f)$ and $\alpha \notin Q_1$. It follows that, for all $f \in S \cap Q_1$, $\ann_B(f) \not\subseteq Q_1$. A similar argument shows that, for all $f \in S \cap Q_1 = S \cap Q_2$, $\ann_B(f) \not\subseteq Q_2$. Hence, $S$ is an OMG-subring of $B$.
%
%
%
\end{proof}

We now use Lemma \ref{primeavoidbig} and Lemma \ref{transcend} to identify elements of $B$ that can be adjoined to a given OMG-subring of $B$ that will result in another OMG-subring of $B$.
%

\begin{lemma}\label{adjoining}
Let $(B, M)$ be a local ring with $B/M$ uncountable. Suppose $\Ass(B)= \{Q_1, \dots, Q_n\}$ with $n \geq 2$. Suppose also that $(R, R\cap M)$ is an OMG-subring of $B$. Let $b \in B$ and let $z \in B$ such that $z \notin Q_1$ and $z \notin Q_2$. Let $J$ be an ideal of $B$ such that $J \nsubseteq Q_1$ and $J \nsubseteq Q_2$. Then there is an element $w \in J$ such that $S = R[b + zw]_{(R[b + zw]\cap M)}$ is an OMG-subring of $B$ with $|S|=|R|$.
\end{lemma}

\begin{proof}
Our proof is heavily based on the proof of Lemma 2.4 in \cite{mg}. Let $i \in \{1, 2\}$, and suppose $b + tz + Q_i=b +t'z +Q_i$ with $t, t' \in B$. Then $z(t-t') \in Q_i$ and since $z \notin Q_i$, we have $t + Q_i = t' + Q_i$. Therefore, $b +tz + Q_i= b +t'z + Q_i$ if and only if $t + Q_i = t' + Q_i$. Let $D_i$ be a full set of coset representatives for the cosets $t + Q_i \in B/Q_i$ that make $b + zt + Q_i$ algebraic over $R/(R \cap Q_i)$. Note that $|D_i| \leq |R|$. Define $D = D_1 \cup D_2$ and $C = \{ Q_1, Q_2\}$. Then $|C \times D| \leq |R| < |B/M|$. By Lemma \ref{primeavoidbig} using $I=J$, there is an element $w \in J$ such that $w \notin \bigcup\{P+r | P \in C, r \in D\}$. Then $b +zw + Q_i$ is transcendental over $R/(R \cap Q_i)$ for $i \in \{1,2\}$. By Lemma \ref{transcend}, $S = R[b +zw]_{(R[b+zw]\cap M)}$ is an OMG-subring of $B$ and $|S|=|R|$.
\end{proof}

Recall that our final ring is to have the same completion as $B$. In order to achieve this, we use the following proposition.



\begin{proposition}[\cite{mg}, Proposition 2.6]\label{completionsame}
Let $(B,M)$ be a local ring and let $T = \widehat{B}$. Suppose $(S,S \cap M)$ is a quasi-local subring of $B$ such that the map $S \longrightarrow B/M^2$ is onto and $IB \cap S = I$ for every finitely generated ideal $I$ of $S$.  Then $S$ is Noetherian and $\widehat{S} = T$.  Moreover, if $B/M$ is uncountable, then $S/(S \cap M)$ is uncountable.
\end{proposition}

In order to apply Proposition \ref{completionsame} and show $\widehat{S}=\widehat{B}$, we need to ensure that the subring $S$ of $B$ contains an element from every coset in $B/M^2$. To do this, we first present the following lemma, which details how to safely adjoin an element of a given coset $b + M^2$ to an OMG-subring and obtain another OMG-subring. We will ultimately adjoin an element from every coset of $B/M^2$.


\begin{lemma}\label{cosetrep}
Let $(B, M)$ be a local ring with $B/M$ uncountable. Suppose that $\Ass(B)= \{Q_1, \dots, Q_n\}$ with $n \geq 2$. Suppose also that $M \neq Q_1$ and $M \neq Q_2$. Let $b \in B$ and suppose $(R, R \cap M)$ is an OMG-subring of $B$. Then there exists an OMG-subring $(S, S \cap M)$ of $B$ such that $R \subseteq S, |S| = |R|$, and $S$ contains an element of the coset $b + M^2$. 
\end{lemma}

\begin{proof}
Our proof is an amended version of the proof of Lemma 2.7 in \cite{mg}. By assumption, $M \neq Q_1$ and $M \neq Q_2$. Hence, $M^2 \not\subseteq Q_1$ and $M^2 \not\subseteq Q_2$.  Use Lemma \ref{adjoining} with $J = M^2$ and $z = 1$ to find $m \in M^2$ such that $S = R[b + m]_{(R[b + m] \cap M)}$ is an OMG-subring of $B$ with $|S| = |R|$.  Note that $R \subseteq S$ and $S$ contains $b + m$, an element of the coset $b + M^2$.
\end{proof}

When using Proposition \ref{completionsame}, we need to ensure that in our final ring $S$, $IB \cap S=I$ for every finitely generated ideal $I$ of $S$. Lemma \ref{closer} will help us do this, the proof of which is largely taken from the proof of Lemma 2.8 in \cite{mg}. Our proof, however, includes the necessary adjustments made to prove the additional annihilator condition of an OMG-subring holds.


\begin{lemma}\label{closer}
Let $(B, M)$ be a local ring with $B/M$ uncountable. Suppose that $\Ass(B)= \{Q_1, \dots, Q_n\}$, with $n \geq 2$. Let $(R, R \cap M)$ be an OMG-subring of $B$. Then, for any finitely generated ideal $I$ of $R$ and for any $c \in IB \cap R$, there is an OMG-subring $(S, S \cap M)$ of $B$ such that $R \subseteq S, |S| = |R|$, and $c \in IS$.
\end{lemma}

\begin{proof}
Let $I = (y_1, y_2, \dots, y_k).$ We induct on $k$.

Let $k=1$. Then $I=aR$ for $a \in R$.  Now let $c \in IB \cap R.$ Then $c=au$ for some $u \in B$. If $a=0$ then $S=R$ works. So assume $a \neq 0$.

First, suppose $a \notin Q_1$. Therefore $a \notin Q_2$. We claim $S=R[u]_{(R[u]\cap M)}$ is the desired OMG-subring of $B$. Suppose $f \in R[u] \cap Q_1$. Then $f = r_mu^m+\dots+r_1u+r_0$ where $r_i \in R$. Hence $a^mf = r_mc^m+r_{m-1}c^{m-1}a+\dots+r_1a^{m-1}+r_0a^m \in R \cap Q_1=R \cap Q_2$ because $a, c \in R$ and $R$ is an OMG-subring. Since $a \notin Q_2$, we have $f \in Q_2$. Hence $R[u]\cap Q_1\subseteq R[u]\cap Q_2$. Similarly, $R[u]\cap Q_2\subseteq R[u]\cap Q_1$ and so $R[u]\cap Q_2 = R[u]\cap Q_1$.  It follows that $S \cap Q_1 = S \cap Q_2.$

We now show that if $f \in R[u] \cap Q_1=R[u] \cap Q_2,$ then $\ann_B(f)\nsubseteq Q_1$ and $\ann_B(f)\nsubseteq Q_2$. Just as before, $a^mf = r_mc^m+r_{m-1}c^{m-1}a+\dots+r_1a^{m-1}+r_0a^m \in R \cap Q_1=R \cap Q_2$. Because $R$ is an OMG-subring, $\ann_B(a^mf) \nsubseteq Q_1$, and so there is a $b \notin Q_1$ such that $ba^mf=0$. Notice that $a \notin Q_1$ so $a^m \notin Q_1$, and we have $ba^m \notin Q_1$. Hence $(ba^m) \in \ann_B(f)$ but $ba^m \notin Q_1$, so we have $\ann_B(f) \nsubseteq Q_1$. Similarly, $\ann_B(f) \nsubseteq Q_2$. It follows that if $f \in S \cap Q_1$ then $\ann_B(f) \not\subseteq Q_1$ and $\ann_B(f) \not\subseteq Q_2$.

Note also that $R \subseteq S, |S|=|R|$, and $c \in IS$ and so $S$ is our desired OMG-subring.

Now assume that $a \in Q_1$. Then $a \in Q_2$. Then $\ann_B(a) \nsubseteq Q_1$ and $\ann_B(a) \nsubseteq Q_2$ because $R$ is an OMG-subring and $a \in R \cap Q_1 = R \cap Q_2$. Using Lemma \ref{adjoining} with $z = 1$, there exists $w \in \ann_B(a)$ such that $S = R[u+w]_{(R[u+w]\cap M)}$ is an OMG-subring of $B$ with $|S|=|R|$. Now, $u + w \in S$ and $a(u+w) = au = c$, and so $c \in IS$. This completes the base case.


Now suppose $k > 1$, and assume that the lemma holds for all ideals generated by fewer than $k$ generators. Note that $c = y_1b_1 + y_2b_2+\dots+y_kb_k$ for some $b_i \in B$.

We first consider the case where $y_i \in Q_1$ for all $i=1,2, \dots, k$. 
 Because $y_1 \in R$, we have $\ann_B(y_1) \nsubseteq Q_1$ and $\ann_B(y_1) \nsubseteq Q_2$ by defintion of OMG-subring. Using Lemma \ref{adjoining} with $J = \ann_B(y_1)$ and $z=1$, we obtain $w \in \ann_B(y_1)$ such that $S'=R[b_1+w]_{(R[b_1+w]\cap M)}$ is an OMG-subring of $B$ with $|S'|=|R|$. Consider the ideal $(y_2, \dots, y_k)$ of $S'$ and let $c^* = c - y_1(b_1+w).$ Then $c^* \in (y_2, \dots, y_k)B \cap S',$ so by our induction hypothesis, there is an OMG-subring $(S, S \cap M)$ of $B$ such that $S' \subseteq S, |S|=|S'|,$ and $c^* \in (y_2, \dots,y_k)S$. Thus $c^* = y_2s_2+ \dots + y_ks_k$ for some $s_i \in S$. Since $c = c^* + y_1(b_1+w),$ we have $c \in (y_1, y_2, \dots, y_k)S = IS$, and it follows that $S$ is the desired OMG-subring of B.  

Now consider the case where $y_i \notin Q_1$ for some $i$. Without loss of generality, suppose $y_2 \notin Q_1$. Then $y_2 \notin Q_2$. Use Lemma \ref{adjoining} with $J=B$ to find $w \in B$ such that $S'=R[b_1+y_2w]_{(R[b_1+y_2w]\cap M)}$ is an OMG-subring of $B$ with $|S'|=|R|$ and $R \subseteq S'$. Note that 
\[c =y_1b_1 + y_1y_2w - y_1y_2w + y_2b_2 + \dots + y_kb_k = y_1(b_1+y_2w) + y_2(b_2 - y_1w)+ \dots + y_kb_k.\]
Consider the ideal $(y_2,\dots, y_k)$ of $S'$ and let $ c^* = c - y_1(b_1+ y_2w).$ Then, $c^* \in (y_2, \dots, y_k)B \cap S'.$ By our induction assumption, there is an OMG-subring $(S, S \cap M)$ of $B$ such that $S' \subseteq S, |S| = |S'|,$ and $c^* \in (y_2, \dots, y_k)S$. So we have $c^* = y_2s_2 + \dots + y_ks_m$ for some $s_i \in S$. Hence, $c = c^* + y_1(b_1 + y_2w) \in (y_1, \dots, y_k)S = IS$, and it follows that $S$ is the desired OMG-subring of B.
\end{proof}

We now present the definition of the earlier referenced Organized OMG-subring.

\begin{definition}\label{2xomg}
Let $(B, M)$ be a local ring with $B/M$ uncountable.  Let $\Ass(B)=\{Q_1, \dots, Q_n\}$ with $n \geq 2$, and suppose that $\height(Q_i)\leq 1$ for all $1\leq i \leq n$, and $M \notin \Ass(B)$. A quasi-local subring $(R, R\cap M)$ of $B$ is called an \emph{Organized One-or-Minimal Gluing Subring} of $B$, or a Double OMG-subring of $B$, if $R$ is an OMG-subring of $B$ satisfying the property that $R \cap Q_i = R \cap Q_j$ if and only if $i=j$ or $\{i,j\}=\{1,2\}$.
\end{definition}

Note that if $(R,R \cap M)$ is an OMG-subring of $B$ that contains a Double OMG-subring of $B$ then $R$ is a Double OMG-subring of $B$. 

The following property provides sufficient conditions on $Q_1$ and $Q_2$ to ensure the existence of a Double OMG-subring of $B$. 

\begin{manualdefinition}{($\ast$)}\label{*}
Let $Q_1$ and $Q_2$ be associated prime ideals of a ring $B$.  We say $Q_1$ and $Q_2$ satisfy property ($\ast$) if they satisfy one of the following conditions

\begin{enumerate}
    \item $\height(Q_1)=\height(Q_2)=0$ 
    \item  $\height(Q_1)=1, \height(Q_2)=0$, and $Q_2$ is the unique minimal prime ideal contained in $Q_1$
\end{enumerate}
\end{manualdefinition}

Lemma \ref{MYLEMMA!}  shows that property \ref{*} is sufficient for constructing a Double OMG-subring of $B$.


\begin{lemma} \label{MYLEMMA!}
Let $(B, M)$ be a local ring with $B/M$ uncountable and $M \notin \Ass(B)$. Let $\Ass(B)= \{Q_1, \dots, Q_n\}$ with $n \geq 2$ and suppose that $\height(Q_i) \leq 1$ for all $1 \leq i \leq n$. Let $(R, R \cap M)$ be an OMG-subring of $B$. If $Q_1$ and $Q_2$ satisfy property \ref{*}, then there exists a Double OMG-subring $(S, S \cap M)$ of $B$ such that $R \subseteq S$ and $|R| = |S|$. 
\end{lemma}

\begin{proof}
If $n = 2$, then $S = R$ works, so assume $n > 2$.
We first organize the height 1 associated prime ideals of $B$. Let 
\[\Ass(B)_1 := \{Q_h | h > 2, ht(Q_h)=1\} = \{J_1, \dots, J_k\}.\]
If $\Ass(B)_1$ is not empty,
let $X =  \Ass(B) \setminus J_1$. Because $\height(J_1)=1$, if $Q \in X$, we have $J_1 \nsubseteq Q$. By the Prime Avoidance Theorem, $J_1 \nsubseteq \bigcup_{Q \in X} Q$. Hence, there exists a $z_1$ such that $z_1 \in J_1$ and $ z_1 \notin Q$ for all $Q \in X$. Note that by assumption, $M \nsubseteq Q_1$ and $M \nsubseteq Q_2$. By Lemma \ref{adjoining}, there is a $w \in M$ such that $S_1 = R[z_1 + z_1w]_{(R[z_1+z_1w] \cap M)}$ is an OMG-subring of $B$ and $|S_1| = |R|$. We observe that $S_1$ contains the element $z_1 + z_1w = z_1(1+w)$. Because $w \in M$, $1 + w$ is a unit of $B$. Hence $z_1(1+w) \notin Q$ for all $Q \in X$ implying that $S_1 \cap J_1 \neq S_1 \cap Q$ for all $Q \in X$. 

Repeat the argument replacing $R$ with $S_1$ and $z_1$ with $z_2$ where $z_2$ is an element of $B$ such that $z_2 \in J_2, z_2 \notin Q$ for all $Q \in \Ass(B)\setminus J_2$. This will generate an OMG-subring $S_2$ where $S_2 \cap J_2 \neq S_2 \cap Q$ for all $Q \in \Ass(B)\setminus J_2$. Continue until we obtain an OMG-subring $S_k$ such that $|S_k| = |R|$ and, if $J_i \in \Ass(B)_1$ and $Q \in \Ass(B)$ with $Q \neq J_i$, then $S_k \cap J_i \neq S_k \cap Q$. If $\Ass(B)_1$ is empty, let $S_k = R$.

We now organize the height 0 associated prime ideals. Let
\[\Ass(B)_0 := \{Q_i | i > 2, ht(Q_i)=0\} = \{I_1, \dots, I_l\}.\]
If $\Ass(B)_0$ is not empty, consider $I_1$ and let $P$ be a minimal prime ideal of $B$ with $P \neq I_1$. Note that $I_1 \nsubseteq P$ and, in particular $I_1 \nsubseteq Q_2$. By assumption, $I_1 \nsubseteq Q_1$ because either $\height(Q_1)=0$ or $Q_2$ is the unique minimal prime ideal contained in $Q_1.$ By the Prime Avoidance Theorem, there exists an element $z'_1$ of $B$ such that $z'_1 \in I_1, z'_1 \notin Q_1$ and if $P$ is a minimal prime ideal of $B$ not equal to $I_1$ then $z'_1 \not\in P$. By Lemma \ref{adjoining}, there is a $w' \in M$ such that $S_{k+1} = S_k[z'_1 + z'_1w']_{(S_k[z'_1+z'_1w'] \cap M)}$ is an OMG-subring of $B$ and $|S_{k+1}| = |R|$. Note that $S_{k+1}$ contains the element $z'_1 + z'_1w' = z'_1(1+w')$. Because $w' \in M$, $1 + w'$ is a unit. Therefore, $z'_1(1+w') \notin Q_1$ and, if $P$ is a minimal prime ideal of $B$ with $P \neq I_1$ then $ z'_1(1+w') \notin P$. So we have $S_{k+1} \cap I_1 \neq S_{k+1} \cap Q_1$ and, if $P$ is a minimal prime ideal of $B$ with $P \neq I_1$, then $S_{k + 1} \cap I_1  \neq S_{k+1} \cap P$. 

Obtain the OMG-subring $S_{k + 2}$ by repeating this argument replacing $S_k$ with $S_{k+1}$ and $z'_1$ with $z'_2$ where $z'_2$ is an element of $B$ such that $z'_2 \in I_2, z'_2 \notin Q_1$ and, if $P$ is a minimal prime ideal of $B$ with $P \neq I_2$ then $z'_2 \not\in P$. Continue until we obtain an OMG-subring $S_{k+l}$. If $\Ass(B)_0$ is empty, let $S_{k + l} = S_k$. Then $|S_{k + l}| = |R|$ and $S_{k+l} \cap Q_i = S_{k + l} \cap Q_j$ if and only if $i=j$ or $\{i,j\}=\{1,2\}$. It follows that $S_{k+l}$ is the desired Double OMG-subring of $B$. 
\end{proof}

The next lemma ensures that, under certain conditions, the union of an increasing chain of Double OMG-subrings is itself a Double OMG-subring. We note that Lemma \ref{unioning} is a modified version of Lemma 2.10 in \cite{mg}.


\begin{lemma}\label{unioning}
Let $(B, M)$ be a local ring with $B/M$ uncountable and $M \notin \Ass(B)$. Let $\Ass(B)=\{Q_1, Q_2, \dots, Q_n\}$ with $n \geq 2$. Suppose $\height(Q_i) \leq 1$ for all $1 \leq i \leq n$. Let $\Omega$ be a well-ordered index set and suppose that $(R_\beta, R_\beta \cap M)$ for $\beta \in \Omega$ is a family of Double OMG-subrings of $B$ such that, if $\alpha, \mu \in \Omega$ with $\alpha < \mu,$ then $R_\alpha \subseteq R_{\mu}$. Then $S=\bigcup_{\beta \in \Omega}R_\beta$ is an infinite subring of $B$ such that for all $a \in R \cap Q_1$, $\ann_B(a) \not\subseteq Q_1$, $\ann_B(a) \not\subseteq Q_2$ and such that $S \cap Q_i = S \cap Q_j$ if and only if $i=j$ or $i, j \in \{1,2\}$. Furthermore, if there is some cardinal $\lambda < |B/M|$ such that $|R_\beta| \leq \lambda$ for all $\beta \in \Omega$, and if $|\Omega| < |B/M|$, then $|S| \leq \max\{ \lambda, |\Omega|\}$ and $S$ is a Double OMG-subring of $B$.
\end{lemma}

\begin{proof}
Since $R_{\beta}$ is infinite for all $\beta \in \Omega$, $S$ is infinite. Given some $a \in S \cap Q_1,$ we know that $a \in R_\beta \cap Q_1$ for some Double OMG-subring $R_\beta$ of $B$ and so $\ann_B(a) \nsubseteq Q_1$ and $\ann_B(a) \nsubseteq Q_2$.
Note that $S \cap Q_i = S \cap Q_j$ if and only if $i=j$ or $i, j \in \{1,2\}$ follows since, for all $\beta \in \Omega$, $R_{\beta} \cap Q_i = R_{\beta} \cap Q_j$ if and only if $i=j$ or $i, j \in \{1,2\}$.


Now suppose there is some cardinal $\lambda < |B/M|$ such that $|R_\beta| \leq \lambda$ for all $\beta \in \Omega$, and $|\Omega|<|B/M|$. Then $S \leq \lambda|\Omega|=\max\{\lambda, |\Omega|\}$. So, $|S| < |B/M|,$ and it follows that $(S, S \cap M)$ is a Double OMG-subring of $B$.
\end{proof}

In Lemma \ref{preordering} we construct a Double OMG-subring of $B$ that satisfies many of our desired properties simultaneously. In order to prove the lemma, the following definition will be helpful.


\begin{definition}
Let $\psi$ be a well-ordered set and let $\alpha \in \psi$. Define $$\gamma(\alpha)=\sup\{\beta \in \psi | \beta < \alpha\}.$$
\end{definition}


\begin{lemma}\label{preordering}
Let $(B, M)$ be a local ring with $B/M$ uncountable and let $\Ass(B)=\{Q_1, Q_2,\dots,Q_n\}$ with $n \geq 2$. Suppose also that $M \not\in \Ass(B)$, $\height(Q_i) \leq 1$ for all $i \in \{1, \dots, n\}$ and $Q_1, Q_2$ satisfy property \ref{*}. In addition, assume $(R, R \cap M)$ is an OMG-subring of $B$ and let $b \in B.$ Then there exists a Double OMG-subring $(S, S \cap M)$ of $B$ such that $R \subseteq S$, $|R|=|S|$, $b + M^2$ is in the image of the map $S \longrightarrow B/M^2,$ and $IB \cap S = I$ for every finitely generated ideal $I$ of $S$.
\end{lemma}

\begin{proof}
First use Lemma \ref{cosetrep} to obtain an OMG-subring $(R', R' \cap M)$ of $B$ such that $R \subseteq R'$, $|R'| = |R|,$ and $R'$ contains an element of $b + M^2$. Next, use Lemma \ref{MYLEMMA!} to get a Double OMG-subring $(R'', R'' \cap M)$ of $B$ such that $R' \subseteq R''$ and $|R''|=|R|'$. Define 
\[ \psi = \{(I, c) | I \ \text{is a finitely generated ideal of}\ R'' \  \text{and} \ c \in IB \cap R''\}.\]
Well-order $\psi$ so that it has no maximal element, and let 0 denote its first element. Note that $|\psi| \leq |R''| = |R|$. We proceed by using transfinite induction. Recursively define a family of Double OMG-subrings $(R_\mu, R_\mu \cap M)$ of $B$ for each $\mu \in \psi$ such that $|R_{\mu}|=|R|$ and, if $\alpha, \rho \in \psi$ with $\alpha < \rho \leq \mu$, then $R_\alpha \subseteq R_{\rho}$. Define $R_0 = R''$. Now, for $\mu \in \psi$ assume that $R_\beta$ has been defined for all $\beta < \mu$ such that $(R_\beta, R_\beta \cap M)$ is a Double OMG-subring of $B, |R_\beta| = |R|,$ and if $\alpha, \rho \leq \beta$ with $\alpha < \rho$, then $R_\alpha \subseteq R_{\rho}$. Suppose $\gamma(\mu) < \mu$, and let $\gamma(\mu) = (I, c).$ Then define $(R_\mu, R_\mu \cap M)$ to be the Double OMG-subring of $B$ obtained from Lemma \ref{closer} such that $R_{\gamma(\mu)} \subseteq R_\mu, |R_{\gamma(\mu)}| = |R_{\mu}|,$ and $c \in IR_{\mu}$. On the other hand, if $\gamma(\mu) = \mu,$ define $R_\mu = \bigcup_{\beta<\mu}R_\beta.$ In this case, by Lemma \ref{unioning}, $(R_\mu, R_\mu \cap M)$ is a Double OMG-subring of $B$ with $|R_\mu| = |R|.$ In either case, we have that $(R_\mu, R_\mu \cap M)$ is a Double OMG-subring of $B$, $|R_\mu| = |R|$, and if $\alpha, \rho \leq \mu$ with $\alpha < \rho$, then $R_\alpha \subseteq R_\rho$.

Let $S_1 = \bigcup_{\mu\in \psi}R_\mu$. By Lemma \ref{unioning}, $(S_1, S_1 \cap M)$ is a Double OMG-subring of $B$ and $|S_1|=|R|.$ Let $I$ be a finitely generated ideal of $R''$ and let $c \in IB \cap R''.$ Then $(I, c) = \gamma(\mu)$ for some $\mu \in \psi$ with $\gamma(\mu) < \mu$. By construction, $c \in IR_\mu \subseteq IS_1.$ It follows that $IB \cap R'' \subseteq IS_1$ for every finitely generated ideal $I$ of $R''$.

Repeat this process with $R''$ replaced by $S_1$ to obtain a Double OMG-subring $(S_2, S_2 \cap M)$ of $B$ with $S_1 \subseteq S_2$, $|S_2| = |R|$, and $IB \cap S_1 \subseteq IS_2$ for every finitely generated ideal $I$ of $S_1$. Continue to obtain a chain of Double OMG-subrings $R'' \subseteq S_1 \subseteq S_2 \subseteq \dots$ with $S_i \subseteq S_{i+1}$, $|S_{i}| = |R|$, and $IB \cap S_i \subseteq IS_{i+1}$ for every finitely generated ideal of $I$ of $S_i$.

Let $S = \bigcup_{i=1}^{\infty} S_i$. By Lemma \ref{unioning}, $(S, S\cap M)$ is a Double OMG-subring of $B$ with $|S|=|R|$. Now suppose $I$ is a finitely generated ideal of $S,$ and let $c \in IB \cap S$. Then $I = (s_1, \dots, s_k)$ for $s_i \in S.$ Choose $N$ such that $c, s_1, \dots, s_k \in S_N.$ Then $c \in IB \cap S_N \subseteq IS_{N+1} \subseteq IS$. It follows that $IB \cap S =I$ and so $S$ is the desired Double OMG-subring of $B$. 
\end{proof}

To apply Proposition \ref{completionsame}, we need the subring $S$ of $B$ to satisfy the conditions that the map $S \longrightarrow B/M^2$ is onto and that $IB \cap S = I$ for every finitely generated ideal $I$ of $S$. In Lemma \ref{ordering}, we construct our subring to satisfy both of these properties.

\begin{lemma}\label{ordering}
Let $(B, M)$ be a local ring with $B/M$ uncountable and let $\Ass(B)=\{Q_1, Q_2,\dots,Q_n\}$ with $n \geq 2$. Suppose also that $M \not\in \Ass(B)$, $\height(Q_i) \leq 1$ for all $i \in \{1, \dots, n\}$ and $Q_1, Q_2$ satisfy property \ref{*}. In addition, assume $(R, R \cap M)$ is an OMG-subring of $B$. Then there exists a subring $(S, S \cap M)$ of $B$ such that $R \subseteq S$, $S \cap Q_i = S \cap Q_j$ if and only if $i=j$ or $i, j \in \{1,2\}$, the map $S \longrightarrow B/M^2$ is onto and $IB \cap S = I$ for every finitely generated ideal $I$ of $S$.
\end{lemma}

\begin{proof}
Let $\Omega = B/M^2$ and well-order $\Omega$ so that each element has fewer than $|B/M^2|$ predecessors. Note that, since $B/M$ is infinite, $|B/M^2| = |B/M|$. Let $0$ denote the first element of $\Omega$ and let $(R_0, R_0 \cap M)$ be the Double OMG-subring obtained from $R$ using Lemma \ref{preordering} with $b = 0$. We recursively define a family $(R_{\beta}, R_{\beta} \cap M)$ for $\beta \in \Omega$ of Double OMG-subrings of $B$, so that $IB \cap R_{\beta} = I$ for every finitely generated ideal $I$ of $R_{\beta}$ and if $\alpha < \rho$, then $R_{\alpha} \subseteq R_{\rho}$. Let $\mu \in \Omega$ and assume that $R_{\beta}$ has been defined for all $\beta < \mu$.  If $\gamma(\mu) < \mu$ then let $\gamma(\mu) = b + M^2$, and define $(R_{\mu}, R_{\mu} \cap M)$ to be the Double OMG-subring of $B$ obtained from Lemma \ref{preordering} so that $R_{\gamma(\mu)} \subseteq R_{\mu}$, $|R_{\gamma(\mu)}| = |R_{\mu}|$, $b + M^2$ is in the image of the map $R_{\mu} \longrightarrow B/M^2$, and $IB \cap R_{\mu} = I$ for every finitely generated ideal $I$ of $R_{\mu}$. If $\gamma(\mu) = \mu$ then define $R_{\mu} = \bigcup_{\beta < \mu}R_{\beta}$. By Lemma \ref{unioning}, $R_{\mu}$ is a Double OMG-subring of $B$. Now suppose $I = (a_1, \ldots, a_k)$ is a finitely generated ideal of $R_{\mu}$, and let $c \in IB \cap R_{\mu}$.  Then there is a $\beta \in \Omega$ with $\beta < \mu$ such that $\{c, a_1, \ldots ,a_k\} \in R_{\beta}$.  
Hence, $c \in IB \cap R_{\beta} = (a_1, \ldots ,a_k)R_{\beta} \subseteq I$. It follows that $IB \cap R_{\mu} = I$ for every finitely generated ideal $I$ of $R_{\mu}$.

Now let $S = \bigcup_{\beta \in \Omega}R_{\beta}$. By construction, the map $S \longrightarrow B/M^2$ is onto, and by the same argument as above, $IB \cap S = I$ for every finitely generated ideal $I$ of $S$.  Finally, Lemma \ref{unioning} gives us that $S \cap Q_i = S \cap Q_j$ if and only if $i=j$.
\end{proof}

We are now ready to state and prove our main result, the Generalized Gluing Theorem. Note that much of the work we have done up to this point has been to identify a local subring $S$ of $B$ possessing the properties necessary to apply Proposition \ref{completionsame}. This proposition allows us to conclude that $\widehat{S}=\widehat{B}$ and $S/(S\cap M)$ is uncountable. We now apply Proposition \ref{completionsame} to do exactly this.


\begin{theorem}(The Generalized Gluing Theorem)\label{bigassthm}
Let $(B, M)$ be a local ring containing the rationals and assume that $B/M$ is uncountable. Suppose that $\Ass(B) = \{Q_1, Q_2, \dots, Q_n\}$ with $n \geq 2$ and that $M \not\in \Ass(B)$. Suppose also that $\height(Q_i) \leq 1$ for all $i \in \{1, 2, \dots, n\}$ and that $Q_1, Q_2$ satisfy property \ref{*}. Then there is a local ring $S \subseteq B$ with maximal ideal $S \cap M$ such that 

\begin{enumerate}
    \item $S$ contains the rationals 
    \item $\widehat{S}=\widehat{B}$ 
    \item $S/(S \cap M)$ is uncountable 
    \item $S \cap Q_i = S \cap Q_j$ if and only if $i=j$ or $i, j \in \{1, 2\}$.
\end{enumerate}

\end{theorem}

\begin{proof}
Note that $R=\Q$ is an OMG-subring of $B$. Now apply Lemma \ref{ordering}, taking $R$ as the starting OMG-subring, to obtain a subring $(S, S \cap M)$ of $B$ such that $R \subseteq S$, $S \cap Q_i = S \cap Q_j$ if and only if $i=j$ or $i, j \in \{1,2\}$, the map $S \longrightarrow B/M^2$ is onto and $IB \cap S = I$ for every finitely generated ideal $I$ of $S$.
By Proposition \ref{completionsame}, $\widehat{S}=\widehat{B}$ and $S/(S \cap M)$ is uncountable. Hence $S$ is the desired subring of $B$.
\end{proof}

\section{Applications and Consequences}\label{Applications}

The Generalized Gluing Theorem (Theorem \ref{bigassthm}) shows that exactly two associated prime ideals of $B$ can be ``glued together'' given that they satisfy the conditions provided in property \ref{*}. Notice that, given an appropriate local ring $B$, if  $B$ has exactly two associated prime ideals, they will necessarily satisfy \ref{*}. We now consider the special case where the local ring $B$ in Theorem \ref{bigassthm} has exactly two associated prime ideals and we show that, in this case, the resulting subring $S$ obtained in Theorem \ref{bigassthm} is a domain. We do this by showing that, using our construction, the ring $S$ in Theorem \ref{bigassthm} satisfies the condition that $S \cap Q_1 = S \cap Q_2 = (0)$. We start with the following crutial observation.

\begin{lemma}\label{transcend-domain}
Let $(B, M)$ be a local ring and let $Q$ be a prime ideal of $B$. Let $(R,R \cap M)$ be a quasi-local subring of $B$ such that $R \cap Q=(0)$. If $x \in B$ such that $x \in B/Q$ is transcendental over $R/(Q \cap R) \cong R$ then $R[x]_{(R[x]\cap M)} \cap Q = (0)$.
\end{lemma}

\begin{proof}
Let $f \in R[x] \cap Q$. Then for some positive integer $m$, 
\[f = r_0 + r_1x + r_2x^2 + \dots + r_mx^m, \text{ with } r_i \in R \text{ for } 0 \leq i \leq m.\]
Since $f \in Q$ and $u + Q$ is transcendental over $R/(Q \cap R),$ $r_i \in Q$ for all $i \in \{0, \dots, m\}$. This means $r_i \in R \cap Q = (0)$ and it follows that $f = 0$.  Therefore, $R[x] \cap Q = (0)$, and it follows that $R[x]_{(R[x]\cap M)} \cap Q = (0)$.
\end{proof}

\begin{remark}\label{zeroint}
A consequence of Lemma \ref{transcend-domain} is that, if the OMG-subring $R$ of $B$ in Lemma \ref{transcend} satisfies $R \cap Q_1 = R \cap Q_2 = (0)$, then the resulting OMG-subring $S$ of $B$ in Lemma \ref{transcend} satisfies the condition that $S \cap Q_1 = S \cap Q_2 = (0)$.  Since Lemma \ref{transcend} is used in the proofs of Lemma \ref{adjoining} and Lemma \ref{cosetrep}, the analogous statement follows for those lemmas as well.  In other words, if the OMG-subring $R$ of $B$ in Lemma \ref{adjoining} (resp. Lemma \ref{cosetrep}) satisfies $R \cap Q_1 = R \cap Q_2 = (0)$, then the resulting OMG-subring $S$ of $B$ in Lemma \ref{adjoining} (resp. Lemma \ref{cosetrep}) satisfies the condition that $S \cap Q_1 = S \cap Q_2 = (0)$.
\end{remark}

\begin{lemma}\label{closer-domain}
Let $(B, M)$ be a local ring with $B/M$ uncountable. Suppose that $\Ass(B)= \{Q_1, Q_2\}$. Let $(R, R \cap M)$ be an OMG-subring of $B$ such that $R \cap Q_1 = R \cap Q_2 = (0)$. Then, for any finitely generated ideal $I$ of $R$ and for any $c \in IB \cap R$, there is an OMG-subring $(S, S \cap M)$ of $B$ such that $R \subseteq S, |S| = |R|$, $c \in IS$, and $S \cap Q_1 = S \cap Q_2 = (0)$.
\end{lemma}

\begin{proof}
  By Lemma \ref{closer} there is an OMG-subring $(S, S \cap M)$ of $B$ such that $R \subseteq S, |S| = |R|$, and $c \in IS$. We have left to show, then, that the $S$ constructed in the proof of Lemma \ref{closer} satisfies the property that $S \cap Q_1 = S \cap Q_2 = (0)$. As most of the cases in the proof of Lemma \ref{closer} use Lemma \ref{adjoining}, we use Remark \ref{zeroint} to conclude that, for those cases, $S \cap Q_1 = S \cap Q_2 = (0)$.
  

  The only case left to consider is when $I$ is generated by a nonzero element $a$ of $R$ where $a \not\in Q_1$. In this case we also have that $a \not\in Q_2$.  Note that, since $Q_1$ and $Q_2$ are the only associated prime ideals of $B$, we have that $a$ is not a zerodivisor in $B$.  We claim that ring given in the proof of Lemma \ref{closer}, $S = R[u]_{(R[u] \cap M)}$ where $c = au$, $u \in B$ is the desired OMG-subring of $B$. We only need show that $S \cap Q_1 = S \cap Q_2 = (0)$. Suppose $f \in R[u] \cap Q_1$. Then $f = r_0 + r_1u + \dots + r_mu^m$ where $r_i \in R$. Hence $a^mf = r_0a^m + r_1a^{m-1} + \dots + r_{m-1}c^{m-1}a + r_mc^m \in R \cap Q_1 = (0)$. As $a$ is not a zerodivisor, $f = 0$ and we have that $R[u] \cap Q_1 = (0)$. Similarly, $R[u] \cap Q_2 = (0)$.  It follows that $S \cap Q_1 = S \cap Q_2 = (0)$.
  %
 %
\end{proof}

\begin{lemma}\label{unioning-domain}
Let $(B, M)$ be a local ring, let $Q_1$ and $Q_2$ be prime ideals of $B$, and let $\Omega$ be a well-ordered index set. Suppose that $(R_\beta, R_\beta \cap M)$ for $\beta \in \Omega$ is a family of subrings of $B$ with $R_\beta \cap Q_1 = R_\beta \cap Q_2 = (0)$ and such that, if $\alpha, \mu \in \Omega$ with $\alpha < \mu,$ then $R_\alpha \subseteq R_{\mu}$. Define $S=\bigcup_{\beta \in \Omega}R_\beta$. Then $S \cap Q_1 = S \cap Q_2 = (0)$.
\end{lemma}

\begin{proof}
Let $x \in S \cap Q_1$.  Then $x \in R_{\beta}$ for some $\beta \in \Omega$.  It follows that $x \in R_{\beta} \cap Q_1 = (0)$.  Hence $S \cap Q_1 = (0)$.  Similarly, $S \cap Q_2 = (0)$.
\end{proof}

\begin{remark}\label{unioningzero}
Note that, by Lemma \ref{unioning-domain}, if the Double OMG-subrings $R_{\beta}$ of $B$ in the statement of Lemma \ref{unioning} all satisfy the condition that $R_{\beta} \cap Q_1 = R_{\beta} \cap Q_2 = (0)$, then the ring $S$ in the conclusion of Lemma \ref{unioning} satisfies the condition that $S \cap Q_1 = S \cap Q_2 = (0)$. As a result of this and Lemma \ref{closer-domain}, if the ring $R$ in the statement of Lemma \ref{preordering} is a Double OMG-subring satisfying the condition that $R \cap Q_1 = R \cap Q_2 = (0)$, then, omitting the use of Lemma \ref{MYLEMMA!} in the proof of Lemma \ref{preordering}, the ring $S$ in the conclusion of Lemma \ref{preordering} also satisfies $S \cap Q_1 = S \cap Q_2 = (0)$. Similarly, if the ring $R$ in the statement of Lemma \ref{ordering} is a Double OMG-subring such that $R \cap Q_1 = R \cap Q_2 = (0)$, then the ring $S$ in the conclusion of Lemma \ref{ordering} satisfies the condition that $S \cap Q_1 = S \cap Q_2 = (0)$.
\end{remark}

We now show that, if $B$ has exactly two associated prime ideals, then the ring $S$ constructed in the proof of Theorem \ref{bigassthm} is a domain.

\begin{theorem}\label{domain}
Let $(B, M)$ be a local ring containing the rationals and assume that $B/M$ is uncountable. Suppose that $\Ass(B) = \{Q_1, Q_2\}$ and that $M \not\in \Ass(B)$. Suppose also that $\height(Q_i) \leq 1$ for all $i \in \{1, 2\}.$ Then there is a local ring $S \subseteq B$ with maximal ideal $S \cap M$ such that
\begin{enumerate}
    \item $S$ contains the rationals
    \item $\widehat{S}=\widehat{B}$ 
    \item $S/(S \cap M)$ is uncountable 
    \item $S \cap Q_1 = S \cap Q_2 = (0)$
\end{enumerate}
In particular, $S$ is a domain.
\end{theorem}

\begin{proof}
In this case, $R = \mathbb{Q}$ is a Double OMG-subring of $B$ with $R \cap Q_1 = R \cap Q_2 = (0)$. By the proof of Theorem \ref{bigassthm} using Lemma \ref{ordering} along with Remark \ref{unioningzero}, the desired ring $S$ exists.
\end{proof}

Of course, Theorem \ref{bigassthm} can also be applied to rings with more than two associated prime ideals. Therefore, Theorem \ref{bigassthm} is a stronger result if it can be iteratively applied. 
In order to show that this is possible, we prove that, in Theorem \ref{bigassthm}, the set of associated prime ideals of $S$ is exactly the set of associated prime ideals of $B$ where $Q_1$ and $Q_2$ have been ``glued'' together. We state this formally after first stating three useful lemmas.

\begin{lemma}[\cite{small2019}, Lemma 2.6]\label{small 2019}
    Let $(T, \widetilde{M})$ be the completion of a local ring $(B, B \cap \widetilde{M})$ and let $P$ be a prime ideal of $B$. If $Q$ is a minimal prime ideal over $PT$ then $Q \cap B = P$.
\end{lemma}

\begin{lemma}\label{going down bs}
Let $(B,M)$ be a local ring and suppose $(S,S \cap M)$ is a local subring of $B$ such that $\widehat{S} = \widehat{B}$. Then the going down property holds between $S$ and $B$.
\end{lemma}

\begin{proof}
  Let $\widehat{S} = \widehat{B} = T$ and suppose $Q_1 \supseteq Q_2$ are prime ideals of $S$. Let $P_1 \in \Spec(B)$ such that $P_1 \cap S = Q_1$. Let $J_1$ be a minimal prime ideal in $T$ over $P_1T$. Then $J_1 \cap B = P_1$ by Lemma \ref{small 2019}. Now $J_1 \cap S = J_1 \cap B \cap S = P_1 \cap S = Q_1.$ Because $T$ is a faithfully flat extension of $S$, the going down property holds and so there exists a prime ideal $J_2$ of $T$ such that $J_1 \supseteq J_2$ and $J_2 \cap S = Q_2$.
 %
Let $P_2 = J_2 \cap B$. Then $P_1 \supseteq P_2$ and $P_2 \cap S = Q_2$.  It follows that the going down property holds between $S$ and $B$.
 %
\end{proof} 

\begin{lemma}\label{asssheight}
 Let $(B,M)$ be a local ring and suppose $(S,S \cap M)$ is a local subring of $B$ such that $\widehat{S} = \widehat{B}$.  Suppose also that, if $Q \in \Ass(B)$, then $\height(Q) \leq 1$.  Then if $P \in \Ass(S)$, we have $\height(P) \leq 1$.
\end{lemma}

\begin{proof}
Suppose $P \in \Ass(S)$. Then $PB$ contains only zero divisors of $B$. Thus, $PB \subseteq Q$ for some $Q \in \Ass(B)$. Now, $\height Q \leq 1$ and so, by Lemma \ref{going down bs}, $\height (S \cap Q) \leq 1$. Therefore, $P \subseteq PB \cap S \subseteq Q \cap S$, and it follows that $\height P \leq 1$.
\end{proof}

We now state and prove the previously mentioned claim about the relationship between the associated prime ideals of $S$ and the associated prime ideals of $B$. 

\begin{proposition}\label{asssprop}
    Let $(B,M)$ be a local ring and suppose $(S,S \cap M)$ is a local subring of $B$ such that $\widehat{S} = \widehat{B}$.  Suppose also that, if $Q \in \Ass(B)$, then $\height(Q) \leq 1$. Then $\Ass(S)= \{Q \cap S \,| \,Q \in \Ass(B)\}.$
\end{proposition}

\begin{proof}
Let $\widehat{S} = \widehat{B} = T$, and suppose $P \in \Ass(S)$. Then $PT$ contains only zerodivisors of $T$, and so $PT$ is contained in an associated prime ideal $I$ of $T$. Let $J$ be a minimal prime ideal of $PT$ contained in $I$.  Then $J$ contains only zerodivisors of $T$ and, by Lemma \ref{small 2019}, $J\cap S = P$. Let $Q = J \cap B$. Then $Q \cap S = P$ and $Q$ contains only zerodivisors of $B$. It follows that $Q \subseteq Q'$ for some $Q' \in \Ass(B)$. Note that, by assumption, $\height (Q') \leq 1$. If $\height (Q) = 0$ then $Q \in \Ass(B)$.  If $\height (Q) = 1$, then $Q = Q' \in \Ass(B)$.  It follows that $\Ass(S) \subseteq \{Q \cap S \,| \,Q \in \Ass(B)\}.$




Now let $Q \in \Ass(B)$.  Then, by a similar argument as in the previous paragraph, there is a prime ideal $J$ of $T$ such that $J$ contains only zerodivisors of $T$ and $J \cap B = Q$. Let $P = Q \cap S$.  Then $P$ contains only zerodivisors of $S$. It follows that $P \subseteq P'$ for some $P' \in \Ass(S)$. By Lemma \ref{asssheight}, $\height P' \leq 1$. If $\height (P) = 0$ then $P \in \Ass(S)$.  If $\height (P) = 1$, then $P = P' \in \Ass(S)$.  Hence $\{Q \cap S \,| \,Q \in \Ass(B)\} \subseteq \Ass(S).$
\end{proof}

Proposition \ref{asssprop} justifies our claim made in the introduction that if the ring $B$ in Theorem \ref{bigassthm} has $n\geq 2$ associated prime ideals then the ring $S$ has $n - 1$ associated prime ideals.  In addition, it justifies our claim in the introduction that we can find local rings $S_1 \supseteq S_2 \supseteq \dots \supseteq S_{n}$, all with the same completion such that, if $\{Q_1, \dots, Q_{n}\}$ are the associated prime ideals of $S_1$, then $S_i$ has exactly $n - i + 1$ associated prime ideals and 
\[\Ass(S_i)=\{S_i \cap Q_1 = S_i \cap Q_2 = \dots = S_i \cap Q_{i}, S_i \cap Q_{i+1}, \dots, S_i \cap Q_{n}\}.\] 

Theorem \ref{bigassthm} provides insight into the relationship between the associated prime ideal structure of a local ring and the associated prime ideal structure of its completion.  Suppose $B$ is a complete local ring satisfying the conditions in Theorem \ref{bigassthm}.  The set $X = \Ass(B)$ is a finite partially ordered set with respect to inclusion.  Theorem \ref{bigassthm} offers a partial answer to the following question.  Suppose that $Y$ is a finite partially ordered set. Under what conditions is there a local ring $S$ with completion $B$ such that $\Ass(S)$, when viewed as a partially ordered set under inclusion, is the same as $Y$?  If all associated prime ideals of $B$ have height less than two, then Lemma \ref{asssheight} shows that $Y$ having dimension at most one is a necessary condition for such a ring $S$ to exist.  Similarly, Proposition \ref{asssprop} shows that $|Y| \leq |X|$ is also a necessary condition.  If $Y$ does satisfy these necessary conditions then Theorem \ref{bigassthm} shows that, if there is a way to construct $Y$ from $X$ by a series of ``gluing'' operations where, at each step, either two minimal nodes are glued together or a minimal node and a height one node where the minimal node is the only minimal node contained in the height one node are glued together, then such a ring $S$ exists. We illustrate with the following example.

\begin{example}
    Let $B=$
        \[\frac{\mathbb{C}[[X_1,X_2,X_3,X_4,X_5,X_6,X_7,X_8,X_9,X_{10}]]}{\langle X_1^2, X_1X_2\rangle \cap \langle X_3^2, X_3X_4 \rangle \cap \langle X_5^2, X_5X_6 \rangle \cap \langle X_7^2, X_7X_8 \rangle \cap \langle X_9^2, X_9X_{10} \rangle}\]
The ring $B$ has 10 associated prime ideals, with 5 of them minimal and the other 5 height one. Each minimal prime ideal is contained in exactly one height one associated prime ideal.  Let $Y$ be the partially ordered set consisting of 5 elements where one is minimal, four are height one, and the minimal node is contained in all of the height one nodes.  Note that it is possible to obtain $Y$ from $X = \Ass(B)$ by first gluing all of the minimal nodes together (two at a time), and then gluing one of the height one nodes to the minimal one.  It follows that Theorem \ref{bigassthm} can be applied to construct a local ring $S$ with completion $B$ such that $\Ass(S)$, when viewed as a partially ordered set, is the same as $Y$.


\end{example}



When applying Theorem \ref{bigassthm} multiple times to construct a chain of subrings of $B$, note that there exists an algorithm that will ensure the process will end with an integral domain.  We first glue all of the minimal prime ideals together (two at a time), and then repeatedly glue a height one associated prime ideal to the minimal prime ideal until there is only one associated prime ideal remaining.  By Theorem \ref{domain}, the last ring in the chain will be an integral domain.


    We end with observations about property \ref{*} that are put on $Q_1$ and $Q_2$ in Theorem \ref{bigassthm}. The next proposition provides insight into why condition (2) is part of property \ref{*}.

\begin{proposition}\label{three}
Let $(B,M)$ be a local ring and suppose $(S,S \cap M)$ is a local subring of $B$ such that $\widehat{S} = \widehat{B}$. Let $Q_2$ be a minimal prime ideal of $B$ and let $Q_1$ be a prime ideal of $B$. If $S \cap Q_2 = S \cap Q_1$ then, for any minimal prime ideal $Q$ of $B$ satisfying $Q \subseteq Q_1,$ we have $S \cap Q = S \cap Q_2 = S \cap Q_1$.
\end{proposition}

\begin{proof}
By Lemma \ref{going down bs}, the going down property holds between $S$ and $B$. As a consequence, since $Q_2$ is a minimal prime ideal of $B$, $S \cap Q_2 = S \cap Q_1$ is a minimal prime ideal of $S$.  Therefore, as $Q \subseteq Q_1$, we have $S \cap Q \subseteq S \cap Q_1$, and it follows that $S \cap Q = S \cap Q_1$.
\end{proof}

Suppose that $Q_1$ is a prime ideal of $B$ of positive height and $Q_2$ is a minimal prime ideal of $B$. A consequence of Proposition \ref{three} is that, to ``glue'' only $Q_1$ and $Q_2$ (with $Q_1$ not glued to any other minimal prime ideals of $B$), it is necessary that $Q_2$ be one of the minimal prime ideals contained in $Q_1$ and, moreover, that it be the only minimal prime ideal contained in $Q_1$. Thus, in the case that $Q_2$ is minimal and $Q_1$ is not, condition (2) of property \ref{*} is necessary in Theorem \ref{bigassthm}.


Our final proposition shows that there are cases where two associated prime ideals of height one cannot be glued together in Theorem \ref{bigassthm}.

\begin{proposition}\label{ht1-no-no}
Let $(B,M)$ be a local ring and suppose $(S,S \cap M)$ is a local subring of $B$ such that $\widehat{S} = \widehat{B}$. Let $Q_1, Q_2\in \Spec(B)$ with $\height(Q_1)=\height(Q_2)=1$ and $S \cap Q_1 = S \cap Q_2$. Suppose $P_2$ is a minimal prime ideal of $B$ such that $P_2 \subseteq Q_2$ and $P_2 \not\subseteq Q_1$. Then there exists a minimal prime ideal $P_1$ of $B$ such that $P_1 \subseteq Q_1$ and $S \cap P_1 = S \cap P_2$.
\end{proposition}

\begin{proof}
Note that $S \cap P_2 \subseteq S \cap Q_2 = S \cap Q_1$.  Since by Lemma \ref{going down bs}, the going down property holds between $S$ and $B$, there exists a minimal prime ideal $P_1$ of $B$ such that $P_1 \subseteq Q_1$ and $S \cap P_1 = S \cap P_2$.
%
\end{proof}





Recall that, in Theorem \ref{bigassthm}, once can choose $Q_1$ and $Q_2$ to be any minimal prime ideals of $B$. A consequence of Proposition \ref{ht1-no-no} is that one cannot arbitrarily choose two height one associated prime ideals of $B$ as $Q_1$ and $Q_2$ in Theorem \ref{bigassthm}. In particular, it is a necessary condition in this case that $Q_1$ and $Q_2$ contain exactly the same minimal prime ideals of $B$.





\end{document}